\documentclass[a4paper]{amsart}
\usepackage[T1]{fontenc}

\usepackage{color}
\usepackage{amsxtra}
\usepackage{amsfonts}
\usepackage{amssymb}
\usepackage{latexsym}
\usepackage{amsmath}
\usepackage{amsthm}
\usepackage{mathtools}
\usepackage[mathscr]{eucal}
\usepackage{graphicx}
\usepackage[utf8]{inputenc}
\usepackage{subcaption}

\usepackage[ruled]{algorithm2e}
\usepackage{hyperref}

\usepackage[capitalise]{cleveref}

\usepackage[backend=biber]{biblatex}
\newcommand{\N}{\mathbb{N}}
\newcommand{\R}{\mathbb{R}}

\newcommand{\Z}{\mathbb{Z}}

\newcommand{\COP}{{\mathcal{COP}}}
\newcommand{\CP}{{\mathcal{CP}}}
\newcommand{\Sym}{{\mathcal{S}}}
\newcommand{\Min}{\mathrm{Min}}
\newcommand{\interior}{\mathrm{int}\,}
\newcommand{\cone}{\mathrm{cone}\,}
\newcommand{\SPN}{\mathcal{SPN}}
\newcommand{\NonN}{\mathcal{N}}
\newcommand{\Ryshkov}{\mathcal{R}}

\newcommand{\MinCOP}{\mathrm{Min}_{\COP} \thinspace}
\newcommand{\minCOP}{\mathrm{min}_{\COP} \thinspace}

\theoremstyle{plain}
\newtheorem{thm}{Theorem}[section]
\newtheorem{lemma}[thm]{Lemma}
\newtheorem{coro}[thm]{Corollary}

\theoremstyle{definition}
\newtheorem{defi}[thm]{Definition}
\newtheorem{example}[thm]{Example}

\theoremstyle{remark}

\author{Alexander Oertel}
\address{A.~Oertel, Universit\"at Rostock, Institute of
  Mathematics, 18051 Rostock, Germany}
\email{alexander.oertel@uni-rostock.de}

\author{Achill Sch\"urmann}
\address{A.~Sch\"urmann, Universit\"at Rostock, Institute of
  Mathematics, 18051 Rostock, Germany}
\email{achill.schuermann@uni-rostock.de}

\date{September 27, 2025}

\subjclass{11H50, 11Y16, 15A23, 90C20}

\keywords{Copositive Minimum, Complexity, LDLT-decomposition}

\begin{document}

\title[On Computing the Copositive Minimum and its Representatives]{On Computing the Copositive Minimum and its Representatives}

\begin{abstract}
Computing the copositive minimum of a strictly copositive quadratic form is a natural generalization of computing the arithmetical minimum of a positive definite one.
In this paper we show that this generalized problem is NP-complete.
Moreover, we describe a practical method to calculate all shortest vectors using the LDLT-decomposition in a big class of special cases.
Our numerical tests show that our method performs significantly better than previous approaches.
\end{abstract}

\maketitle

\section{Introduction}
A classical problem in the geometry of numbers with multiple applications in numerous contexts is the calculation of shortest vectors of positive definite quadratic forms (or, equivalently, shortest vectors of a lattice).
More precisely, using matrix notation, the problem is for a positive definite matrix $Q$ to compute the \emph{arithmetical minimum}
\begin{align*}
 \min Q = \min_{z\in \Z^n \setminus \{0\} } Q[z]
 \end{align*}
and the (integral) vectors
\begin{align*}
 \mathrm{Min}\thinspace Q = \{ z \in \Z^n : Q[z] = \min Q \}
\end{align*}
attaining it.
Here we use the notation $Q[z] = z^T Q z$ for column vectors $z \in \R^n$.

We denote by $\Sym^n$ the space of real symmetric $n \times n$ matrices.
A matrix $Q \in \Sym^n$ is said to be \emph{copositive}, if $Q[x] \geq 0$ for all $x \in \R^n_{\geq0}$, and \emph{strictly copositive}, if it is copositive and $Q[x] = 0$ for $x \in \R^n_{\geq 0}$ only if $x = 0$.
For a matrix $Q \in \Sym^n$ the \emph{copositive minimum} is a natural generalization of the arithmetical minimum and defined as
\begin{align*}
 \minCOP Q := \min \left\{ Q[z] : z \in \Z^n_{\geq 0} \setminus \{0\} \right\}.
\end{align*}
The vectors attaining it are denoted by
\begin{align*}
 \MinCOP Q := \{ z \in \Z^n_{\geq 0} : Q[z] = \minCOP Q \}.
\end{align*}
We call the task of computing $\MinCOP Q$, i.e. enumerating all nonnegative shortest vectors, the \emph{MinCOP problem}.
Note that every calculation of a classical arithmetical minimum and the corresponding shortest vectors can be realized as a copositive minimum and representatives calculation \autocite[Theorem~6.1]{PerfectCopositive}.

These notations have been introduced in \autocite{CertificateAlgo}, in the context of computing certificates for complete copositivity of a given matrix, which we summarize briefly in \cref{sec:origin}.
The main bottleneck there is the computation of the copositive minimum.
The latter is deeply connected to the copositivity of $Q$.
\begin{lemma}\label{lemma:cop_min_strictly}
 For a matrix $Q \in \Sym^n$ we have
 \begin{itemize}
  \item[(i)]  $\minCOP Q > 0$ if and only if $Q$ is strictly copositive,
  \item[(ii)] $\minCOP Q = 0$ if and only if $Q$ is copositive, but not strictly so, or
  \item[(iii)] $\minCOP Q = -\infty$ if and only if $Q$ is not copositive.
 \end{itemize}
\end{lemma}
\begin{proof}
 Part (i) follows from \autocite[Lemma~2.3]{CP_Factorization_First} and \autocite[Lemma~2.2]{CertificateAlgo}.
 If $z \in \Z_{\geq 0}^n$ with $Q[z] < 0$, then $Q[\alpha z] = \alpha^2 Q[z] \to -\infty$ for $\alpha \to \infty$ with $\alpha \in \Z$, thus (iii).
 The remaining part (ii) follows from exhaustion of all cases.
\end{proof}
Lemma 2.2 of \autocite{CertificateAlgo} also yields that the copositive minimum of a strictly copositive matrix is attained, but attained only by finitely many vectors.
Strictly copositive matrices are thus a natural domain for the copositive minimum, though for mostly technical reasons we allow general symmetric matrices as an input.

We give a short example which we pick up repeatedly.
\begin{example}\label{ex:cop}
 The matrix
\begin{align*}
 Q = \begin{pmatrix}
      3 & -1 & 3 \\
      -1 & 2 & -1 \\
      3 & -1 & 2
     \end{pmatrix}
  \end{align*}
  is strictly copositive, since
  \begin{align*}
   Q[x] = 2 x_1^2 + (x_1 - x_2)^2 + 6x_1 x_3 + (x_2 - x_3)^2 + x_3^2.
  \end{align*}
  A quick $\mathrm{mod} \thinspace 2$ argument shows that $Q[z] > 1$ for $z \in \Z^n_{\geq 0} \setminus \{0\}$. Then we have $\minCOP Q = 2$ and
  \begin{align*}
   \MinCOP Q = \left\{ \begin{pmatrix}
                          0 \\ 1 \\ 0
                         \end{pmatrix}, \,
                         \begin{pmatrix}
                          0 \\ 1 \\ 1
                         \end{pmatrix}, \,
                         \begin{pmatrix}
                          0 \\ 0 \\ 1
                         \end{pmatrix}
                          \right\}
  \end{align*}
  are the vectors attaining the copositive minimum.
  \end{example}

    The computational ansatz for MinCOP utilized in \parencite{CertificateAlgo} is based on a simplicial cone partition of the nonnegative orthant such that $v^T Q w > 0$ for each pair $v$, $w$ of extreme ray generators of a common simplicial cone in $\R_{\geq 0}^n$.
    This simplicial subdivision of $\R_{\geq 0}^n$ was first introduced by Dür und Bundfuss \autocite{BundfussDuer} in the context of testing for copositivity.
  Intuitively, the condition translates to the matrices being nonnegative with respect to each cone separately and in each cone a Fincke-Pohst-type algorithm (see \cref{sec:svp_fincke}) is performed.
  This elegant approach has the drawback that some matrices require a seemingly large number of cones.
  For this reason there is much performance improvement to be gained by finding alternative ways to calculate the copositive minimum.

Copositivity itself has numerous applications and connections to other parts of mathematics, cf. \autocite[Section~2.1]{CopositiveAndCompletelyPositiveMatrices}.
Arguably the most important one is copositive optimization.
Many NP-complete problems have reformulations in this context and the already mentioned certificates are a central building block here for verifiable computer proofs, which are becoming more and more important.
For details on copositive optimization we refer to the surveys \autocite{DuerCopOptSurvey} and \autocite{3e5818f03f824fc9b887acab6834a300}.

Our paper is structured as follows.
We start in \cref{sec:Complexity} by establishing that (a decision variant of) the copositive minimum is NP-complete.
In \cref{sec:Preliminaries} the necessary background information about symmetric matrix cones, the origins of the copositive minimum, the Fincke-Pohst algorithm for the shortest vector problem, and the LDLT-decomposition is then reviewed.
Next, in \cref{sec:Main} the LDLT-decomposition is combined with the Fincke-Pohst algorithm to attack MinCOP.
In particular, we discuss how the decomposition should be used and introduce the concept of \emph{difficult coordinates}, which is central for our new algorithm.
In \cref{sec:special_cases} we describe how our method may solve multiple special cases, such as the case of one difficult coordinate, semidefinite matrices, and when the matrix is decomposable into a positive semidefinite and nonnegative part.
Our method is then analyzed numerically in \cref{sec:appl_and_num_res}.
Finally, several open questions and possible directions for further research are discussed in \cref{sec:open_questions}.
\section{Complexity of the Copositive Minimum}\label{sec:Complexity}
In this section we prove that a decision variant of the copositive minimum is NP-complete by reduction of the \emph{Subset Sum Problem} (see \autocite{GareyJohnson}) to it.
The analogue \emph{shortest vector problem (SVP)} is known to be NP-hard under randomized reductions \autocite{10.1145/276698.276705}, while NP-hardness under deterministic reductions is still an open problem.
Its complexity in general is still an active research topic (cf. e.g. the recent survey \autocite{SurveyBennett}).
We remark that the rather small symmetry group of the copositive cone (see \cref{sec:pre_matrix_cones}) induces a lack of (a comparable) reduction theory for copositive matrices.
This hints that calculating the copositive minimum has a much harder practical difficulty, which we also observe in practice.
\begin{defi}[Copositive Minimum -- Decision Variant]
 Given a strictly copositive matrix $Q$ with integer entries and a positive number $\lambda$, decide if there is $z \in \Z_{\geq 0}^n \setminus \{0\}$ with $Q[z] \leq \lambda$.
\end{defi}
\begin{defi}[Subset Sum Problem]
 Given a vector $a$ of dimension $n$, consisting of positive integers and a target positive integer $s$, decide if there is $x \in \{0, 1\}^n$ such that $a^T x = s$.
\end{defi}
The proof of the NP-completeness of the Subset Sum Problem given in \autocite[Thm.~8.23]{AlgorithmDesign} yields that this variant of the subset sum problem (with only positive integers) is NP-complete.
\begin{thm}
 The decision variant of the copositive minimum is NP-complete.
\end{thm}
\begin{proof}
Note that a certificate for $\minCOP Q \leq \lambda$ would be a suitable vector $z \in \Z_{\geq 0}^n \setminus \{0\}$: We can calculate $Q[z]$ in polynomial time and check the inequality in polynomial time as well.

Let $(a;s)$ be an instance of the subset sum problem.
Define
\begin{equation*}
 Q := n \begin{pmatrix}
         a a^T & -sa \\
         -sa^T & s^2
        \end{pmatrix} +
        \begin{pmatrix}
         2 I_n & -\boldsymbol{1}_n \\
         - \boldsymbol{1}_n^T & n
        \end{pmatrix},
\end{equation*}
with $I_n$ being the identity matrix of dimension $n$ and $\boldsymbol{1}_n$ the all-ones column vector of dimension $n$.
The matrix $Q$ can be calculated in polynomial time in the input length.

A later part of the proof will yield that $Q[y] \geq n$ for all $y \in \Z_{\geq 0}^n\setminus \{0\}$ from which we deduce $Q \in \interior \COP^{n+1}$ by \cref{lemma:cop_min_strictly}.
Alternatively, this follows from
\begin{equation*}
 Q = H^T H + \begin{pmatrix}
              I_n & \\
              & 0
             \end{pmatrix}
\end{equation*}
for
\begin{equation*}
 H = \begin{pmatrix}
      \sqrt{n} a^T & -\sqrt{n} s \\
      -I_n & \boldsymbol{1}_n
     \end{pmatrix}.
\end{equation*}

We write $y = \left( \begin{smallmatrix} x \\ t \end{smallmatrix} \right)$ with $x \in \R^n$ and $t \in \R$.
Then we have
\begin{equation*}
 Q[y] = n \left( (a^T x)^2 - 2sta^T x + t^2 s^2 \right) + 2x^T x - 2t \boldsymbol{1}_n^T x + t^2 n.
\end{equation*}
We also abbreviate
\begin{equation*}
 h_t(x) := (a^T x)^2 - 2sta^T x + t^2 s^2 = (a^T x - ts)^2 \geq 0.
\end{equation*}

Now we show that the subset sum problem has a solution $x\in \{0, 1\}^n$ if and only if the copositive minimum of $Q$ is less than or equal to $n$.

If $x \in \{0, 1\}^n$ is a solution to the subset sum problem, i.e. $a^T x = s$, then
\begin{align*}
 Q\left[ \begin{pmatrix} x \\ 1 \end{pmatrix}\right] &= n (s^2 - 2s^2 + s^2) + 2 x^T x - 2 \boldsymbol{1}_n^T x + n \\
 &= 2  \sum_{i = 1}^n (x_i^2 - x_i) + n = n,
\end{align*}
since $x_i^2 - x_i = 0$ for $x_i \in \{0, 1\}$.

In the other direction, we show that if the subset problem has no solution, then the copositive minimum of $Q$ is strictly greater than $n$ (and is in general greater than $n$).
To that matter, we consider three cases for the last coordinate $t$ of $y \in \Z_{\geq 0}^{n+1} \setminus \{0\}$.

\emph{Case 1:} $t = y_{n+1} \geq 2$.
Since $x_i^2 - tx_i \geq -\frac{t^2}{4}$, we can bound
\begin{align*}
 Q[y] &= n h_t(x) + 2 x^T x - 2t \boldsymbol{1}_n^T x + t^2 n \\
 &=n h_t(x) + 2 \sum_{i=1}^n (x_i^2 - t x_i) + t^2 n \\
 &\geq n h_t(x) + 2n \frac{-t^2}{4} + t^2 n \\
 &= n h_t(x) + \frac{t^2}{2} n > n.
\end{align*}

\emph{Case 2:} $t = y_{n+1} = 1$.
We have
\begin{equation*}
 Q[y] = n h_1(x) + 2 \sum_{i=1}^n (x_i^2 - x_i) + n.
\end{equation*}
If $x \in \{0,1\}^n$, then $h_1(x) > 0$, because there is no solution for the subset sum problem.
If $x \notin \{0, 1\}^n$, then $\sum_{i=1}^n (x_i^2 - x_i) > 0$.
So in total, $Q[y] > n$.

\emph{Case 3:} $t = y_{n+1} = 0$.
Then $x \neq 0$ and thus $a^T x \geq 1$.
In particular,
\begin{equation*}
 Q[y] = n (a^T x)^2 + 2 x^T x > n. \qedhere
\end{equation*}
\end{proof}
\section{Preliminaries and Notations for Our Algorithms}\label{sec:Preliminaries}
\subsection{Symmetric Matrix Cones}\label{sec:pre_matrix_cones}
Throughout the paper we equip the space $\Sym^n$ with the standard trace scalar product $\langle \cdot, \cdot \rangle$.
In this section we introduce different convex cones contained in $\Sym^n$ and refer to \autocite{CopositiveAndCompletelyPositiveMatrices} for a thorough treatment.
Arguably the most important convex cone contained in $\Sym^n$ is the \emph{positive semidefinite cone} $\Sym_{\succcurlyeq 0}^n$ consisting of the positive semidefinite matrices of dimension $n$.
The cone is generated by the rank $1$ matrices, i.e.
\begin{equation*}
 \Sym_{\succcurlyeq 0}^n = \cone \{ xx^T : x \in \R^n \}
\end{equation*}
and is selfdual, meaning that
\begin{equation*}
 (\Sym_{\succcurlyeq 0}^n)^D = \{Q \in \Sym^n : \left\langle Q, X \right \rangle \geq 0 \text{ for all } X \in \Sym_{\succcurlyeq 0}^n \} = \Sym_{\succcurlyeq 0}^n.
\end{equation*}

Another cone that has been investigated for a long time is the \emph{completely positive cone}
\begin{equation*}
 \CP^n = \cone\{ x x^T : x \in \R_{\geq 0}^n\}.
\end{equation*}
See \autocite[Section~3.1]{CopositiveAndCompletelyPositiveMatrices} for a short survey on the numerous applications of the cone.
Its dual cone is the \emph{copositive cone}
\begin{align*}
 \COP^n &=\{Q \in \Sym^n : Q[x] \geq 0 \text{ for all } x \in \R_{\geq 0}^n\} \\
 &= \{Q \in \Sym^n : \left\langle Q, X \right\rangle \geq 0 \text{ for all } X \in \CP^n \}.
\end{align*}
Its interior $\interior \COP^n$ is precisely made up of the strictly copositive matrices \autocite[Theorem~2.25]{CopositiveAndCompletelyPositiveMatrices}.
As seen in \cref{lemma:cop_min_strictly}, a matrix $Q \in \COP^n$ is strictly copositive if its copositive minimum is strictly greater than zero.
The automorphisms $\Phi$ preserving $\COP^n$ are of the form
\begin{equation*}
 \Phi(Q) = (PD)^T Q PD
\end{equation*}
with a permutation matrix $P$ and a diagonal matrix $D$ with only positive entries \autocite{Automorphisms1, Automorphisms2}.
Matrices of the form $PD$ are called \emph{monomial}.

The copositive cone contains $\Sym_{\succcurlyeq 0}^n$ and also the nonnegative cone $\NonN^n$ of all matrices with nonnegative entries.
Therefore, we also have
\begin{equation*}
 \SPN^n = \Sym_{\succcurlyeq 0}^n + \NonN^n \subseteq \COP^n.
\end{equation*}
Up to $n = 4$ the cones $\SPN^n$ and $\COP^n$ coincide \autocite{Diananda_1962}.
For $n \geq 5$ the matrices in $\COP^n \setminus \SPN^n$ are called \emph{exceptional} and a standard example is the so called \emph{Horn matrix} (see \autocite{Diananda_1962} for its first appearance).
Some information regarding the relative size of $\SPN^n$ in $\COP^n$ can be found in \autocite{Klep, ConstructionExceptional}.

The dual cone of $\SPN^n$ is the cone of the doubly nonnegative matrices
\begin{equation*}
 \mathcal{DNN} = \Sym_{\succcurlyeq 0}^n \cap \NonN^n,
\end{equation*}
which we list here for completeness only.

  \subsection{Certificates for Complete Positivity}\label{sec:origin}
  In this section we sketch how the copositive minimum originated in \autocite{CertificateAlgo} from the study of $\CP$-factorizations.
  In doing so we also introduce the \emph{perfect copositive matrices}, which are an important subclass of copositive matrices to test our algorithms against.
  This section is not strictly necessary to understand the algorithms themselves.

  The usual definition for a matrix to be completely positive, i.e. $Q \in \CP^n$, is that there exists a nonnegative matrix $B$ (that is, $B$ has nonnegative entries) such that $Q = B^T B$.
  An equivalent condition is the existence of a $\CP$-factorization
  \begin{equation*}
   Q = \sum_{i = 1}^m \alpha_i x_i x_i^T \text{ with } m \in \N,\, \alpha_i \geq 0, \, x_i \in \R_{\geq 0}^n \text{ for } i = 1, \ldots, m.
  \end{equation*}
  Such a $\CP$-factorization is called \emph{rational}, if the coefficients $\alpha_i$ and the (components of the) vectors $x_i$ are rational for all $i$.
  It has been shown that deciding whether $Q \in \CP^n$ is NP-hard~\autocite{Dickinson2014}.

  Nonetheless, in some important contexts, such as computer proofs based on results from copositive optimization, it is essential to find certificates for (non-) complete positivity (cf. \autocite{BermanDuerShakedMondererOpenQuestions}).

  In \autocite{CertificateAlgo} a simplex type algorithm for rational $\CP$-factorizations is described and in this context the notion of the copositive minimum is introduced.
  It is used to define the \emph{copositive Ryshkov polyhedron}
  \begin{equation*}
   \Ryshkov := \{ Q \in \Sym^n : \minCOP Q \geq 1\},
  \end{equation*}
  which is a so called locally finite polyhedron (see \autocite{CompGeometryForms} for details about locally finite polyhedra).
  It is on the edge graph of this set where the simplex type algorithm is performed.

  In analogy to the classical theory of perfect quadratic forms, we call copositive matrices \emph{perfect copositive} if they are uniquely determined by their copositive minimum and its representatives.
  In particular, the vertices of $\Ryshkov$ are precisely the perfect copositive matrices with copositive minimum equal to $1$.

  The algorithm now walks along a path of edges in $\Ryshkov$, visiting neighboring perfect copositive matrices $B^{(k)}$ (vertices of $\Ryshkov$), until $Q$ is either contained in the (inner) normal cone
  \begin{equation*}
   \mathcal{V} (B^{(k)}) = \cone \left\{v v^T : v \in \MinCOP B^{(k)} \right\}
  \end{equation*}
  (dubbed \emph{Voronoi cone}) or a perfect copositive matrix $B^{(k)}$ acts as a separating witness, i.e. is a certificate for non-complete positivity via $\langle Q, B^{(k)} \rangle < 0$.

  Finding the next contiguous perfect copositive matrix in $\Ryshkov$ requires a large number of copositive minimum calculations:
  For a given perfect copositive matrix $Q$ one drops (at least) one minimal vector, obtains a ray on which the next perfect copositive matrix has to be and performs a search on the ray.
  The copositive minimum for each matrix on the ray then yields the information whether one is already beyond the next perfect copositive matrix or not.

  Note that there are some other algorithms for checking for complete positivity.
  Some recent ones are by Lai and Yoshise~\autocite{Lai2022} and Badenbroek and de Klerk~\autocite{badenbroek2020analyticcentercuttingplane}.
  In these references one can also find further references to other approaches.
  See also~\autocite{CertificateAlgo} for a division into numerical and theoretical algorithms.
  For more information on perfect copositive matrices we refer to \autocite{PerfectCopositive}, where first theoretical investigations into these matrices are given.

\subsection{Shortest Vector Problem and Fincke-Pohst-Algorithm}\label{sec:svp_fincke}
Given a lattice~$\Lambda$ (in the sense of an additive and discrete subgroup of an Euclidean vector space) with a fixed basis we denote its positive definite Gram matrix by $Q$.
The shortest vector problem now asks for the minimal length of a non-zero lattice vector.
We refer to the survey \autocite{YasudaSurvey} for information about the problem, in particular how it is solved.
See also \autocite{MicciancioWalterEnumeration}.
Since we are interested in all shortest vectors for MinCOP, we describe in this section the core idea behind enumeration algorithms.

The usual approach to enumerate all shortest vectors is to shift the viewpoint towards the positive definite quadratic form~$Q$, where the lattice minimum translates to the (square root of the) arithmetical minimum and the minimal vectors of~$Q$ correspond to coefficients of linear combinations giving short lattice vectors.

Next we describe the rough idea of the Fincke-Pohst algorithm \autocite{FinckePohst}, which is the usual way to compute $\Min \thinspace Q$ and also the fundamental idea we build upon in this paper.

The goal of the Fincke-Pohst algorithm for a given positive definite matrix~$Q$ and a constant $\lambda > 0$ is to enumerate all integral points inside the ellipsoid
\begin{equation*}
 \{x \in \R^n : Q[x] \leq \lambda \}.
\end{equation*}
The tool to do so is the \emph{Lagrange expansion}
\begin{equation*}
 Q[x] = \sum_{i = 1}^n A_i \left( x_i + \sum_{j = i + 1}^n u_{ij} x_j \right)^2
\end{equation*}
of $Q[x]$ with the \emph{outer coefficients} $A_i > 0$ and the \emph{inner coefficients} $u_{ij} \in \R$.

This expansion is equivalent to the Cholesky-decomposition of $Q$.
If $Q$ is given as the Gram matrix of some basis $B = (b_1, \ldots, b_n) \in \R^{n\times n}$, then there is also a simple geometric interpretation involving the Gram-Schmidt-orthogonalization $\bar{B} = (\bar{b}_1, \ldots, \bar{b}_n)$.
If $B$ is given as $B = \bar{B} U$, then
\begin{equation*}
 Q[x] = x^T B^T B x = (Ux)^T \bar{B}^T \bar{B} (Ux).
\end{equation*}
In particular, the outer coefficients $A_i$ are the squared lengths of the $\bar{b}_i$.
The inner coefficients are given as
\begin{equation*}
 u_{ij} = \frac{\langle b_i, \bar{b}_j\rangle}{\langle \bar{b}_j, \bar{b}_j \rangle}
\end{equation*}
and describe the ``relative length'' of the projection of $b_i$ onto the line spanned by~$\bar{b}_j$.

This special structure allows for a recursive algorithm, in which each coordinate is subsequently bounded.
More precisely, if $Q[x] \leq \lambda$ is wanted, then the inequality $Q[x] \geq A_n x_n^2$ implies
\begin{equation*}
 - \sqrt{\frac{\lambda}{A_n}}\leq x_n \leq \sqrt{\frac{\lambda}{A_n}}.
\end{equation*}
Since $x$ is to be integral, there are only finitely many values $x_n$ can take.
If we fix $x_n$ to any such value, the inequality
\begin{equation*}
 \lambda - A_n x_n^2 \geq A_{n-1} ( x_{n-1} + u_{n-1\, n}\ x_n)^2
\end{equation*}
yields a range in which $x_{n-1}$ has to be in.

In the same way one can obtain bounds for each coordinate $x_i$ if $x_{i+1}, \ldots, x_n$ are fixed.
This yields a search tree overall, whose leaves (at the lowest level) correspond to the desired integral vectors.
When only the shortest vectors are of interest, one can lower the constant $\lambda$ with each new leaf corresponding to a shorter vector, effectively pruning the search tree.
\subsection{LDLT-Decomposition}\label{sec:ldlt}
In this subsection we describe for a given symmetric matrix its \emph{LDLT-decomposition}, which is a generalization of the Cholesky-decomposition and also the central tool in the remainder of this paper.
For self-containment we also describe how to calculate the decomposition as far as we need it.
\begin{defi}
 Let $Q \in \Sym^n$. An \emph{LDLT-decomposition} of $Q$ is a decomposition
 \begin{align*}
  Q = L D L^T,
 \end{align*}
 with a lower triangular matrix $L$ with unit diagonal and a block diagonal matrix~$D$, where the blocks are at most of size $2$.
\end{defi}
A classical application of the decomposition is the calculation of the inertia of a symmetric matrix and solving linear symmetric systems \autocite{BunchKaufman, BunchParlett}.
For a treatment other than the just mentioned references, see e.g. \autocite[Chapter~4]{MatrixComputations}.
It is a fact that for a symmetric matrix $Q \in \Sym^n$ there exists a permutation matrix $P$, such that the symmetric permutation $P^T Q P$ of $Q$ has an LDLT-decomposition.
The use of $2\times2$ blocks is necessary both for existence and numerical stability, with the standard example for a matrix without an LDLT-decomposition with a proper diagonal matrix being
\begin{equation*}
 \begin{pmatrix}
 0 & 1 \\
 1 & 0
 \end{pmatrix}.
\end{equation*}
See also \cref{sec:NoBlocks}.
Positive semidefinite matrices, however, always admit an LDLT-decomposition with no permutation needed (see e.g. \autocite[Theorem~2.2.6]{SemidefiniteApproachGraphRealization}).

We omit the details here, because for our application only LDLT-decompositions with a proper diagonal matrix $D$ are (currently) interesting to us.
With such decompositions, we obtain a \emph{Lagrange-type expansion}
\begin{align}
 Q[x] =  \sum_{i = 1}^n D_{ii} \left(x_i + \sum_{j = i + 1}^n u_{ij} x_j \right)^2, \label{eq:lagrange_type}
\end{align}
with $(u_{ij}) := U := L^T$, with the important difference that the \emph{outer} coefficients $D_{ii}$ may be zero or negative.
We can tolerate the drawbacks associated with disallowing $2\times2$ blocks in the diagonal matrix, because in \cref{sec:NoBlocks} we describe a strategy that deals with nonexistence.
All (relevant) calculations can be made in exact (rational) arithmetic so that numerical instability is not an issue either.

Next we describe next how to perform an LDLT-step (with proper diagonal matrix).
We partition the matrix $Q$ into
\begin{align*}
 Q = \begin{pmatrix}
      E & C^T \\
      C & B
     \end{pmatrix},
\end{align*}
with $E = (A_{11})$, $C \in \R^{(n-1) \times 1}$, and $B \in \R^{(n-1) \times (n-1)}$.
We then have $Q = L A L^T$ with
\begin{align*}
 L = \begin{pmatrix}
      1 & 0 \\
      C E^{-1} & I_{n-1}
     \end{pmatrix} \text{ and }
     A = \begin{pmatrix}
          E & 0 \\
          0 & B - C E^{-1} C^T
         \end{pmatrix},
\end{align*}
provided that $E \neq 0$.
We note that this step is compatible with permutations, which fix the first coordinate:
If we have
\begin{align*}
 P = \begin{pmatrix}
      1 & \\
      & \tilde{P}
     \end{pmatrix}
\end{align*}
with $\tilde{P}$ being a $(n-1)\times(n-1)$-permutation matrix, we have that applying the step to $P^T Q P$ yields $\tilde{P}^T (B - C E^{-1} C^T )\tilde{P}$ as the remainder.

This one step is equivalent to \emph{completing the square} for the first coordinate, for example
\begin{align}\label{eq:completing_square}
 a x_1^2 + 2b x_1 x_2 + c x_2^2 &= a \left( x_1 + \frac{b}{a}x_2 \right)^2 + \left( c - \frac{b^2}{a} \right) x_2^2.
\end{align}

One can then iterate the step with the remaining matrix $B - C E^{-1} C^T$ to obtain an LDLT-decomposition of $Q$.
\begin{example}\label{ex:ldlt}
 The matrix $Q$ from \cref{ex:cop} has the LDLT-decomposition
\begin{align*}
 Q =         \begin{pmatrix}
         1 & 0 & 0 \\
         -\frac{1}{3} & 1 & 0 \\
         1 & 0 & 1
        \end{pmatrix} \,
        \begin{pmatrix}
         3 & 0 & 0 \\
         0 & \frac{5}{3} & 0 \\
         0 & 0 & -1
        \end{pmatrix} \,
        \begin{pmatrix}
         1 & 0 & 0 \\
         -\frac{1}{3} & 1 & 0 \\
         1 & 0 & 1
        \end{pmatrix}^T.
\end{align*}
\end{example}

\section{Applying LDLT to the MinCOP Problem}\label{sec:Main}
\subsection{Difficult Coordinates}
Let in this section $Q \in \Sym^n$ and let $Q^{(k)} \in \Sym^k$ be the leading principal submatrix of size $k$.
Knowing an LDLT-decomposition of $Q$ means knowing it for $Q^{(k)}$ as well.
\begin{lemma}\label{lemma:pd_submatrix}
Suppose $Q = L D L^T$ is an LDLT-decomposition of $Q$.
Then for $1 \leq k \leq n$ and
$x = (x_1 \, \ldots \, x_k)^T$
the Lagrange expansion
\begin{align*}
  Q^{(k)} [x] = \sum_{i = 1}^k D_{ii} \left( x_i + \sum_{j = i + 1}^k u_{ij} x_j \right)^2,
\end{align*}
with $U = (u_{ij}) = L^T$, holds.
\end{lemma}
\begin{proof}
 By definition we have
 \begin{align*}
  Q[(x_1 \, \ldots \, x_k \, 0 \, \ldots \, 0)^T] = \sum_{i,j  = 0}^k q_{ij} x_i x_j = Q^{(k)}[x]
 \end{align*}
 and by the Lagrange expansion \cref{eq:lagrange_type} of $Q$ we have
 \begin{align*}
  Q[(x_1 \, \ldots \, x_k \, 0 \, \ldots \, 0)^T] &= \sum_{i = 1}^k D_{ii} \left( x_i + \sum_{j = i + 1}^k u_{ij} x_j \right)^2. \qedhere
 \end{align*}
\end{proof}
\begin{coro}
 Let $1\leq k \leq n$ and assume $D_{ii} > 0$ for $1 \leq i \leq k$.
 Then $Q^{(k)}$ is positive definite.
\end{coro}

Our idea is now to replace $Q$ by $P^T Q P$ for a permutation matrix $P$, such that
\begin{equation*}
 k := \mathrm{argmin}_i \{ D_{ii} \leq 0 \}
\end{equation*}
 is as large as possible.
This is useful, since if $x_k, \, \ldots, \, x_n$ are assumed to be fixed, we can obtain bounds on $x_1, \, \ldots\,, x_{k-1}$ implied by $Q[x] \leq \lambda$ from the expansion in the same way as in the positive definite case.
We provide an example, to make clear that the ``easy'' coordinates $x_1$, $\ldots$, $x_{k-1}$ can be handled as usual.

\begin{example}\label{ex:ldlt_bounds}
We continue with the matrix $Q$ from \cref{ex:cop} and \cref{ex:ldlt}.
Let us search for vectors $x \in \Z^n_{\geq 0}$ satisfying $Q[x] \leq 2$. \cref{ex:cop} showed that $2$ is the copositive minimum of $Q$.
Setting $x_3 = 0$ to start, we find
\begin{align*}
 2 \geq 3\left(x_1 - \frac{1}{3}x_2\right)^2 + \frac{5}{3}x_2^2.
\end{align*}
This is the exact setting discussed in \cref{sec:svp_fincke}, so we may continue as usual, finding the vector $(0\,1\,0)^T$.
Setting $x_3 = 1$ on the other hand, we find
\begin{align*}
 3 \geq 3 \left(x_1 - \frac{1}{3} x_2 + 1 \right)^2 + \frac{5}{3} x_2^2,
\end{align*}
which allows us to continue in the usual way as well, leading to the vectors $(0\ 0\ 1)^T$ and $(0\,1\,1)^T$.
In \cref{ex:bounds_for_single_diff} we find that these are all shortest vectors by showing that $x_3$ may only take the values $0$ and $1$.
\end{example}
\begin{defi}
 Let $Q = LDL^T$ be an LDLT-decomposition of $Q \in \Sym^n$ and let $k = \mathrm{argmin}_{i} \, \{D_{ii} \leq 0 \}$.
 Then the coordinates $x_k,\, \ldots,\, x_n$ are called \emph{difficult coordinates}.
\end{defi}

A reformulation of the previously stated main idea is to minimize the number of difficult coordinates.
By Sylvester's law of inertia the number of negative (zero) diagonal entries of $D$ is exactly the number of negative (zero) eigenvalues.
This implies that the number number negative and zero eigenvalues is a lower bound for the number of difficult coordinates.
\cref{lemma:pd_submatrix} implies now the following corollary.
\begin{coro}
 The minimal number of difficult coordinates (of any symmetric permutation of $Q$) is the same as the minimal number of rows and columns that are necessary to delete to obtain a positive definite matrix.
\end{coro}

To conclude, only bounds for the difficult coordinates are necessary.
These bounds do not need to be particularly good, since the bounds of the ``easy'' coordinates can mitigate the bad bounds.
A rather simple idea is to ``box'' the difficult coordinates in:
We calculate for a difficult coordinate $x_k$ the value of (or find a bound for)
\begin{align}
 \max \{x_k : x \in \R_{\geq 0}^n, Q[x] \leq \lambda \}. \label{eq:max_xk}
\end{align}
This bound ignores the other coordinates, so it might not be very good.
Even worse, the optimization problem is not convex and very difficult to solve in general.
In \cref{sec:special_cases} we give some cases in which this is, however, easy to solve.

Geometrically speaking, we intersect the set $\{x \in \R^n : Q[x] \leq \lambda\}$ consecutively with suitable affine subspaces such that the intersection is a (lower-dimensional) ellipsoid, on which we can apply the Fincke-Pohst algorithm again.
See \cref{fig:geometric_idea}.
\begin{figure}[ht]
 \centering
 \includegraphics[width=5cm]{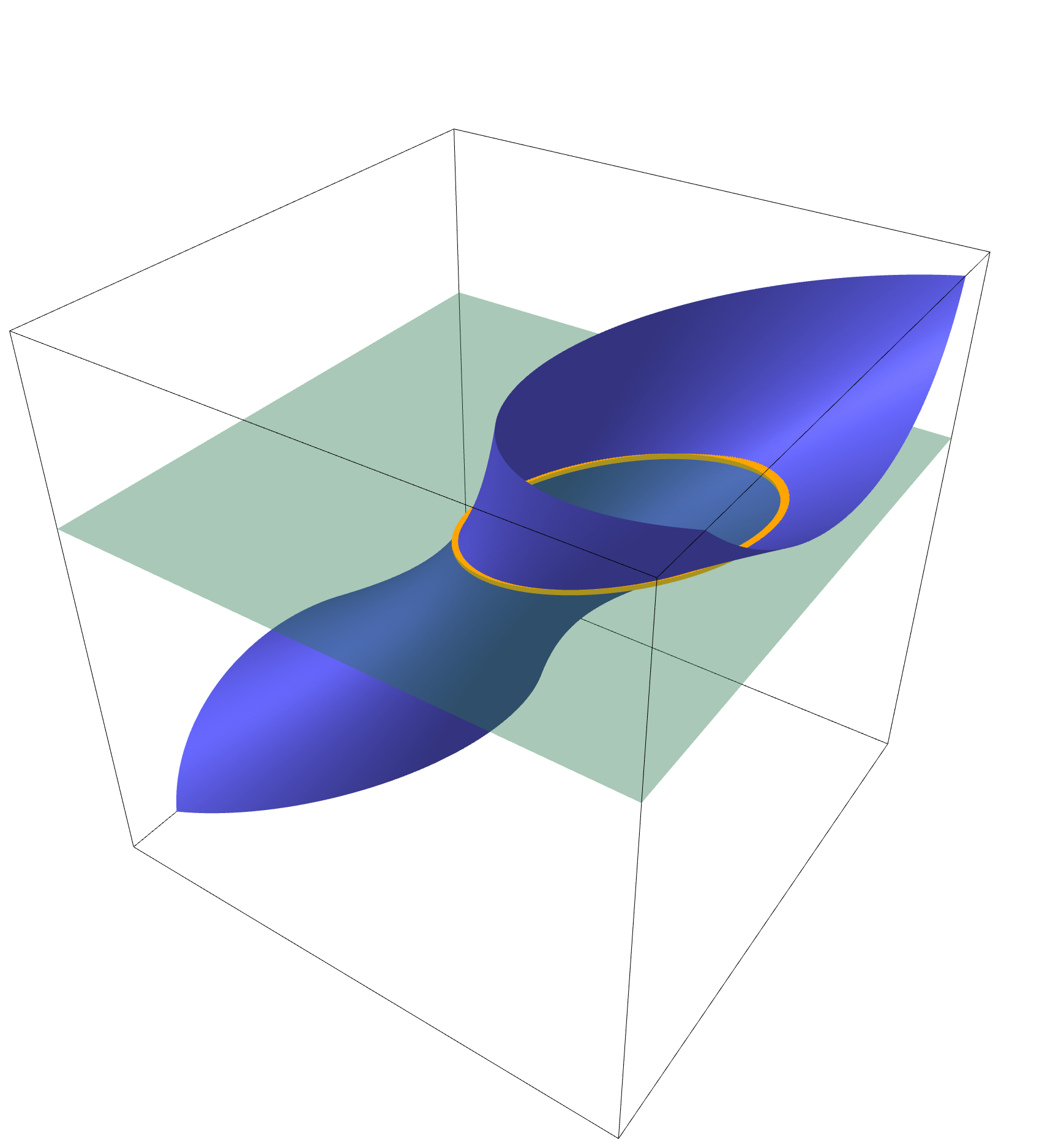}
 \caption{A surface plot of $Q[x] = 2$ with $Q$ from \cref{ex:cop} with the plane $x_3 = 1$.}
 \label{fig:geometric_idea}
\end{figure}

\subsection{Pivot Strategy}
In this subsection we explain how to accomplish the goal formulated in the last subsection heuristically.
Finding the permutation of a suitable symmetric permutation $P^T Q P$ can can be done while performing the LDLT-algorithm by choosing a \emph{pivot element} among the diagonal elements.
With respect to this pivot element the LDLT-step is performed.
In practice, this means applying a transposition so that the chosen diagonal element becomes the new first diagonal element.
Then a regular LDLT-step can be performed as described in \cref{sec:ldlt}.
With the viewpoint of completing the square, as in \cref{eq:completing_square}, the pivot element decides towards which variable we complete the square.

Our heuristic to obtain the permutation is split into two phases.
In the first, we aim to, figuratively speaking, push negative outer coefficients as far right as possible.
Towards this, we always choose the largest positive diagonal element as the pivot.

In the second phase we reorder the earlier ``easy'' coordinates to improve the performance of the algorithm.
This procedure has also already been proposed by Fincke and Pohst (cf. also~\autocite[Sec.~2.7.3]{CohenComputational}):
A reordering is made, with the idea being that a segment $(x_k,\, \ldots\,, x_n)$ should be extendable to a vector $(x_1, \, \ldots\,, x_n)$, since we have a certain amount of vectors in the set
\begin{equation*}
 \{z \in \Z^n_{\geq 0} : Q[z] \leq \lambda \},
\end{equation*}
which need to be traversed in the search tree.
Better extendability means traversing fewer ``useless'' branches of the tree.
In our case the reordering translates to having the (positive) diagonal elements of $D$ be increasing in value.

One way to accomplish this is described next.
We take the LDLT-decomposition
\begin{align*}
 P_1^T Q P_1 = L D L^T \text{ and } (P_1^T Q P_1)[x] = \sum_{i = 1}^n D_{ii} \left(x_i + \sum_{j = i + 1}^n u_{ij} x_j \right)^2
\end{align*}
of (the symmetric permutation of) $Q$.
Then we subtract the summands corresponding to the difficult coordinates.
More precisely, define
\begin{align*}
 R[x] &\coloneqq (P_1^T Q P_1)[x] - \sum_{i = k + 1}^n D_{ii} \left(x_i + \sum_{j = i + 1}^n u_{ij} x_j \right)^2\\
 &= \sum_{i = 1}^{k} D_{ii} \left(x_i + \sum_{j = i + 1}^n u_{ij} x_j \right)^2,
\end{align*}
where $k + 1$ is the smallest difficult coordinate.
The matrix $R$ is given by $L \tilde{D} L^T$, where $\tilde{D}$ arises from $D$ by setting
\begin{align*}
\tilde{D}_{k+1\,k+1} = \ldots = \tilde{D}_{nn} = 0.
\end{align*}
Notice that by its LDLT-decomposition the matrix $R$ is positive semidefinite (and has a kernel of dimension $n-k$).
Furthermore, $R$ is of the form
\begin{equation*}
 \begin{pmatrix}
  \tilde{R} &  \\
   & 0
 \end{pmatrix}
\end{equation*}
with $\tilde{R} \in \Sym^{k}$.
Now we apply the LDLT-algorithm to $R$, where we use the smallest non-zero diagonal element as the pivot element, while we fix the last $k$ coordinates.
This means that we skip the last $n - k$ diagonal elements in the search for the pivot element.
Or, in other terms, we calculate an LDLT-decomposition of $\tilde{R} \in \Sym^{k}$ taking the smallest non-zero diagonal as the pivot element.
This leads to
\begin{align*}
 A := P_2^T R P_2 = L_2 D_2 L_2^T.
\end{align*}
Note that the last $n-k$ diagonal elements of $D_2$ are $0$ by the choice of $P_2$ and $R$.

Overall, it follows that
\begin{align*}
 Q[P_1 P_2 x] = A[x] + \sum_{i = k + 1}^n D_{ii} \left( x_i + \sum_{j = i + 1}^n u_{ij} x_j \right)^2
\end{align*}
is a Lagrange expansion, where the outer coefficients are heuristically positive as long as possible as well as small for the first few coefficients.

\subsection{Avoiding the Need for Block Diagonal LDLT}\label{sec:NoBlocks}
Sometimes an LDLT-decomposition with a proper diagonal matrix $D$ is not possible if at some point no non-zero diagonal entries remain.
The standard example for this case is $\left( \begin{smallmatrix} 0 & 1 \\ 1 & 0 \end{smallmatrix} \right)$.
This matrix (and its multiples) are the only two dimensional examples.
In three dimensions, for $x, y, z \neq 0$ the matrices
\begin{align*}
 \begin{pmatrix}
  0 & x & y \\
  x & 0 & z \\
  y & z & 0
 \end{pmatrix},
 \begin{pmatrix}
  0 & 0 & x \\
  0 & y & 0 \\
  x & 0 & 0
 \end{pmatrix},
 \begin{pmatrix}
  \pm\frac{xy}{z} & x & y \\
  x & \pm \frac{xz}{y} & z \\
  y & z & \pm \frac{yz}{x}
 \end{pmatrix},
 \end{align*}
  where at least one $\pm$ is $-$, and their symmetric permutations are the only examples.
  This is apparent by performing an LDLT-step with every possible pivot element and comparing the resulting $2\times2$ matrix with the two dimensional counterparts.
  Note that none of these matrices are strictly copositive.

  In dimension four we can do a similar procedure, comparing the matrices obtained by doing an LDLT-step to the three dimensional matrices, which do not admit an LDLT-decomposition with a proper diagonal matrix.
  This leads to several systems of polynomial equations.
  In particular, we find matrices which are both strictly copositive and need blocks in their LDLT-decomposition, for example
  \begin{align*}
   \begin{pmatrix}
    8 & -2 & -8 & 0 \\
    -2 & 1 & 0 & 8 \\
    -8 & 0 & 24 & 16 \\
    0 & 8 & 16 & 32
   \end{pmatrix}.
  \end{align*}
  This indicates that we have to take special care of such matrices.

  A simple idea to deal with such matrices is to split off a small nonnegative part, i.e. writing
  \begin{align*}
   Q = \tilde{Q} + N,
  \end{align*}
  with $N \in \NonN^n$ such that $\tilde{Q}$ can be handled as usual.
  Note that this can be always accomplished if $Q$ is strictly copositive.
  The examples in two and three dimensions as well as our calculations for the four dimensional case (which is omitted here) indicate that such matrices are very specific.
  Thus by splitting something off, one should expect to obtain a matrix which can be handled without block matrices.
  Since $Q$ is strictly copositive the diagonal elements of $Q$ are a natural choice to take from.
  In \cref{sec:spn_case} we describe how to find the necessary bounds in the case of ``split'' matrices $\tilde{Q} + N$.

  To close this section, we list some existence results for LDLT-decompositions without $2\times2$ blocks.
  From the Cholesky-decomposition it follows that all positive (and negative) semidefinite matrices possess such decompositions.
  More generally, every symmetric permutation of a \emph{quasidefinite matrix} possesses one as well \autocite{Quasidefinite}, i.e. every matrix of the form
  \begin{equation*}
   \begin{pmatrix}
    -E & A^T \\
    A & F
   \end{pmatrix},
  \end{equation*}
  where $E$ and $F$ are positive definite matrices.
  The existence of the decomposition $Q = LDL^T$ (without permutations) is known to be equivalent to the existence of an LU-factorization (cf. \autocite[Sec.~4.1.2]{MatrixComputations}).
  Necessary and sufficient conditions for that were given by Okunev and Johnson \autocite{okunev2005necessarysufficientconditionsexistence}.

\section{Special Cases}\label{sec:special_cases}
\subsection{One Difficult Coordinate}\label{sec:one_difficult_coord}
This special case can be readily solved by additionally solving just one convex quadratic optimization problem to obtain bounds for the difficult coordinate.
To this end we partition
\begin{align*}
 Q = \begin{pmatrix}
      Q^{(n-1)} & q \\
      q^T & Q_{nn}
     \end{pmatrix} \text{ and }
     x =
     \begin{pmatrix}
      \hat{x} \\
      x_n
     \end{pmatrix}
\end{align*}
with $\hat{x} \in \R^{n-1}$ and $ Q^{(n-1)} \in \Sym^{n-1}$.
If we fix $x_n$, the function
\begin{align*}
 Q[x] = Q^{(n)}[\hat{x}] + 2 x_n q^T \hat{x} + Q_{nn} x_n^2
\end{align*}
is a convex quadratic function in $\hat{x}$.
Indeed, $Q^{(n - 1)}$ is positive definite, since $x_n$ is the only difficult coordinate.

We can obtain bounds then as follows.
Define
\begin{align*}
 f(s) := \min \left\{ Q[x] : x \in \R_{\geq 0}^n,\, x_n = s \right\} \negmedspace.
\end{align*}
Note that for $z \in \Z^n_{\geq 0}$, if $f(z_n) > \lambda$, then $Q[z] > \lambda$ as well.
Thus for $Q[z] \leq \lambda$ to hold, $f(z_n) \leq\lambda$ is a necessary condition.
This observation allows us to bound the difficult coordinate, as the map $f$ is very well-behaved.
\begin{lemma}
 For $s \geq 0$ we have $f(s) = s^2 f(1)$.
\end{lemma}
\begin{proof}
 Clearly, $f(0) = 0$, so assume $s > 0$.
 Then
 \begin{align*}
  f(s) &= \min \left\{ Q[ ( x_1 \, \ldots \, x_{n-1} \, s )^T] : x_1,\ldots, x_{n-1} \geq 0 \right\} \\
  &= s^2 \min \left\{ Q\left[ \left( \frac{x_1}{s} \, \ldots \, \frac{x_{n-1}}{s} \medspace 1 \right) \right] : x_1,\ldots, x_{n-1} \geq 0 \right\} = s^2 f(1). \qedhere
 \end{align*}
\end{proof}
The value $f(1)$ can be calculated by solving a single convex quadratic minimization problem, which can be done efficiently using different methods, see e.g.~\autocite{Boyd_Vandenberghe_2004}.
The difficult coordinate $x_n$ may then be bounded as $0 \leq x_n \leq \sqrt{\frac{\lambda}{f(1)}}$.

Note that this technique may also be applied to positive semidefinite matrices~$Q$:
We fix a difficult coordinate $x_k = 1$, calculate the minimum (which is easy since the corresponding submatrix is positive semidefinite as well!), and finally obtain the bound analogously.

To close this section, we continue with the Examples \ref{ex:cop}, \ref{ex:ldlt}, and \ref{ex:ldlt_bounds}.
\begin{example}\label{ex:bounds_for_single_diff}
 To find bounds on $x_3$ we minimize
 \begin{align*}
  Q[(x_1\, x_2\, 1)^T] = 3x_1^2 - 2x_1x_2 + 6x_1 + 2x_2^2 -2x_2 + 2
 \end{align*}
 over $\R^2_{\geq 0}$.
 The optimal solution is $(0,\frac{1}{2})$ with the objective value $f(1) = \frac{3}{2}$.
 Taking $\lambda = 2$ then yields $0 \leq x_3 \leq \sqrt{\frac{4}{3}}$, so only the cases $x_3 = 0$ and $x_3 = 1$ have to be considered.
\end{example}

\subsection{Matrices in SPN}\label{sec:spn_case}
For this section we assume $Q \in \SPN^n$, e.g. $Q = S + N$ with $S \in \Sym_{\succcurlyeq 0}^n$ and $N \in \NonN^n$.
Suppose now that we have fixed $x_{k+1}, \ldots, x_n$ already and that we want to bound $x_k$ next.
A simple way to find bounds for the difficult coordinates is to lower bound $Q[x]$ by $S[x]$ and apply the ideas of \cref{sec:one_difficult_coord} to the positive semidefinite matrix $S$.
This is what we currently do.

Another idea is as follows.
We can compute an LDLT-decomposition $S = LDL^T$ for $S$ as well (even without permutations since it can be derived from the Cholesky-decomposition of $S$), so that from the quadratic inequality
\begin{align*}
 \lambda \overset{!}{\geq} Q[x] \geq \sum_{i = k}^n D_{ii} \left(x_i + \sum_{j = i + 1}^n u_{ij} x_j\right)^2 + \sum_{i = k}^n \sum_{j = k}^n N_{ij} x_i x_j
\end{align*}
(in $x_k$) we can find bounds for the difficult coordinate $x_k$.
However, if the rank of $S$ is low and (some of) the coefficients of $N$ are also small, then the inequality may produce inacceptably large bounds.
This seems to be not uncommon, which is why we prefer the first approach.
A combination, depending on the given matrix might be advantageous.

It is usually not clear whether a given strictly copositive matrix is in $\SPN^n$.
This can be tested, however, efficiently by a semidefinite program:
The matrix $Q$ is in $\SPN^n$ if and only if there is nonnegative matrix $N$ such that $Q - N \succcurlyeq 0$.
Note that $N \in \NonN^n$ is a linear condition in $\Sym^n$.
Thus by solving the problem
\begin{align*}
 \sup &\left\langle O, N \right\rangle \\
 \mkern8mu \text{s.t. }  &N \in \NonN^n \\
 \mkern30mu &Q - N \succcurlyeq 0
\end{align*}
where $O \in \Sym^n$ is a suitable objective, we can find such a decomposition if it exists.
For the objective $O$ we propose the matrix made up entirely of $-1$ entries.
The effect we wish to achieve with this objective is that the nonnegative part $N$ has only a small influence, so that bounding only via $S$ is not too bad.

Since the copositive cone and $\SPN^n$ coincide for $n \leq 4$, our new method is always applicable for $n \leq 4$.
In general, this method works for all non-exceptional copositive matrices.

\section{Numerical Results}\label{sec:appl_and_num_res}

We implemented the algorithm with the two special cases to bound the difficult coordinates.
It is a preliminary implementation using Python \autocite{Python}, numpy \autocite{harris2020array}, Gurobi \autocite{gurobi}, cvxpy \autocite{diamond2016cvxpy, agrawal2018rewriting}, CVXOPT \autocite{cvxopt}, and OSQP \autocite{osqp}.
The solution described in \cref{sec:NoBlocks} is not yet implemented.
Note that our implementation (the bound calculation specifically) relies currently on non-rigorous floating point arithmetic.
The code and the matrices used for testing are publicly available \autocite{MinCOP_LDLT}.

We have also tested our algorithms on multiple classes of matrices in dimensions three to eight, with $15$ matrices for each class and dimension.
The classes are positive semidefinite matrices, $\SPN$-matrices, $\SPN$-matrices with at least two negative eigenvalues, and, lastly, perfect copositive matrices.
The matrices of the last class have been manually sampled in the early stages of the certificate algorithm of \autocite{CertificateAlgo}.
The other matrices have been randomly sampled as $B B^T$ from Sage's \autocite{sagemath} random\_matrix.
We recorded the runtime and applicability of our algorithms.
For comparison we also tested the implementation of the algorithm using a simplicial partition, introduced in \autocite{CertificateAlgo}.
In particular, we used \autocite{Mathieu}, compiled on July 25, 2025.
If the execution time, however, exceeded two hours for a single matrix, we aborted the calculation and used two hours as the recorded time.
We also recorded the execution time of our algorithms from the moment the necessary libraries (e.g. numpy, gurobipy, cvxpy) were loaded, since in many examples, these loading times are significantly longer than the actual code execution time and would not be present in a more thorough implementation.
The tests were done on a local machine with an Intel Core i7-1185G7 with 32GiB of RAM running Ubuntu 24.04.2.
In \cref{fig:results} the results are displayed.

It is clear that our new methods offer a tremendous speedup in dimensions five and above, if they are applicable.
They also appear to be a lot more consistent and predictable in their runtime in comparison to the method on simplicial subdivisions.
Furthermore, they appear to scale better to higher dimensions.
\begin{figure}[ht]
 \begin{subfigure}{0.48\textwidth}
  \includegraphics[width=\textwidth]{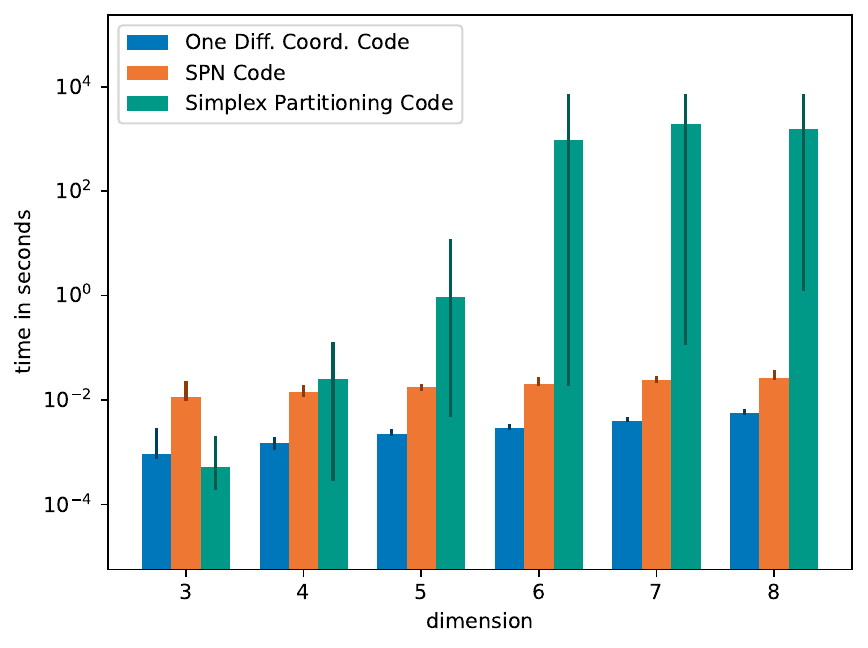}
  \caption{Test on positive semidefinite matrices. All examples could be solved by our algorithms.}
 \end{subfigure}
 \quad
 \begin{subfigure}{0.48\textwidth}
  \includegraphics[width=\textwidth]{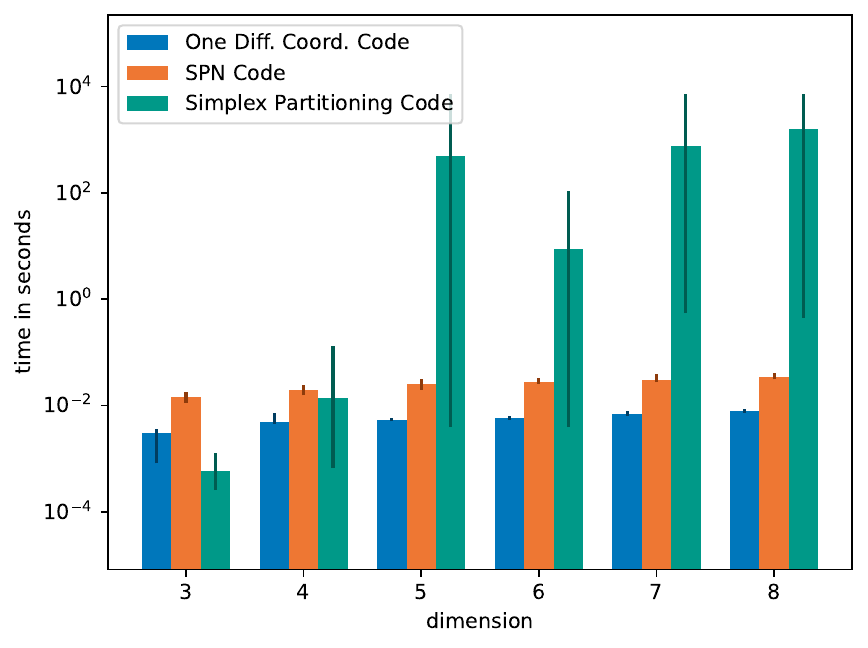}
  \caption{Test on $\SPN$ matrices. The one difficult coordinate code was not applicable for 2 of the 90 examples.}
 \end{subfigure}
 \begin{subfigure}{0.48\textwidth}
  \includegraphics[width=\textwidth]{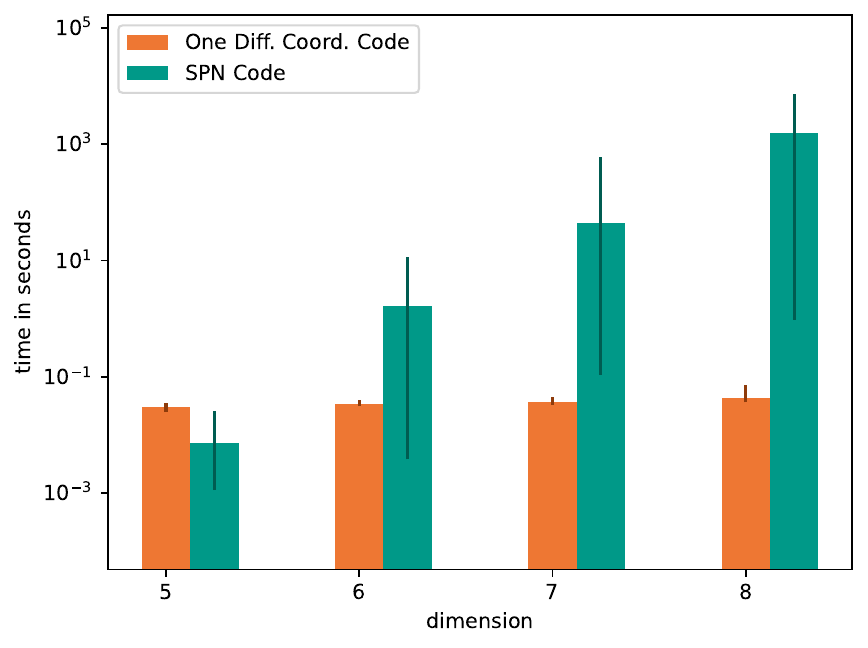}
  \caption{Test on $\SPN$ matrices with at least two negative eigenvalues. The one difficult coordinate code was not applicable here.}
 \end{subfigure}
 \quad
 \begin{subfigure}{0.48\textwidth}
  \includegraphics[width=\textwidth]{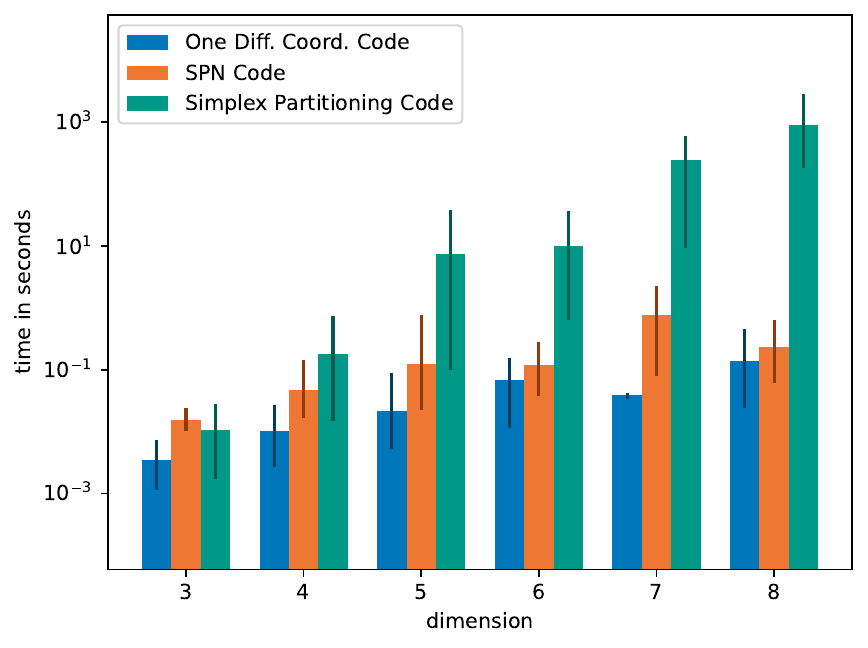}
  \caption{Test on perfect copositive matrices. The one difficult coordinate code was not applicable for 22 of the 90 examples.}
 \end{subfigure}
 \caption{Runtime of the algorithms on a logarithmic timescale. The broad bar indicates the average runtime of the algorithm on the 15 examples of the dimension, the lower end of the small bar indicates the minimum runtime of the algorithm on the 15 examples of that dimension, and the upper end of the small bar indicates the maximum runtime of the algorithm on the 15 examples of that dimension.
 Overall, the code for $\SPN$-matrices was always applicable, while the code for one difficult coordinate matrices was not applicable for 84 out of 359 examples.}
 \label{fig:results}
\end{figure}

\section{Open Questions and Further Directions}\label{sec:open_questions}
In this final section we discuss some open questions and possible further research directions.

\subsection{Extensions to All Copositive Matrices}
\cref{sec:special_cases} was largely concerned with finding good approximations for \cref{eq:max_xk}.
If one had a decent and reliable way to find such upper bounds, the approach using an LDLT-decomposition would work for all input matrices in general.
However, due to the non-convex nature of the problem, this has turned out to be very difficult.
For example, the standard Shor (SDP) relaxation turns out to be unbounded.

A promising approach seems to be to relax the problem
\begin{align*}
 \max \left\{ x_k^2 : x \geq 0, \, Q[x] \leq 1 \right\}
\end{align*}
to
\begin{align*}
  &\max \left\{ \left\langle O, x x^T \right\rangle : x\geq 0, \left\langle Q, x x^T \right\rangle \leq 1 \right\} \\
 \leq &\sup \left\{ \left\langle O, Y \right\rangle : Y \in \CP^n, \, \left\langle Q, Y \right\rangle \leq 1  \right\},
\end{align*}
where $O$ is a matrix everywhere zero except at $O_{kk} = 1$.
This is a in general NP-hard problem and solving it exactly becomes quickly intractable.
A sufficient relaxation might, however, be easier to find here.

A second approach is to combine our $\SPN$ special case with the simplicial partition used in \autocite{CertificateAlgo}.
Some preliminary tests indicate that the number of simplicial cones necessary for our method to become applicable is generally quite low.
Further, one could replace the semidefinite problems we are solving with the approach of \autocite{KaplanCopositivityProbe} to calculate the $\SPN$ decomposition.

\subsection{Reduction Theory}\label{sec:reduction}
One of the highlights of the classical theory of quadratic forms is its reduction theory.
Likewise, the much celebrated \emph{LLL-algorithm} \autocite{Lenstra1982} is a lattice algorithm, which allows, for instance, for a $2^n$ approximation to SVP in polynomial time.
This algorithm is indispensable in the computational aspects of lattice theory and a significant question is if something similar can be done in the copositive situation as well.
This would allow for efficient preprocessing, similar to what is done for the shortest vector problem.

The major hurdle is that by replacing $Q \in \interior \COP^n$ with $U^T Q U$ for an unimodular matrix $U$, we also change the cone in which we search for the shortest vectors.
In particular, we have to search for integer vectors in the polyhedral cone
\begin{align*}
 C := \{x \in \R^n : Ux \geq 0\}.
\end{align*}
The problem of having a different cone already appears in \autocite{CertificateAlgo} and is solved there by an application of the Hermite normal form.
By the unimodularity of $U$ this is in our case the same as reversing the reduction.

Another drawback is that it is not clear how to apply the important special case of \cref{sec:spn_case}.

Special care has also be put into how the reduction is done.
The idea to split the quadratic form $Q$ into a positive and negative part, so that $Q = Q^+ - Q^-$, and applying the LLL-algorithm to $Q^+ + Q^-$ seems to increase the number of difficult coordinates.
A more promising approach is the use of an adapted \emph{indefinite-LLL-algorithm} \autocite{SimonIndefiniteLLL}, which works for indefinite quadratic forms as well by adjusting the so called Lov\'asz condition to disallow negative numbers.

\subsection{Special Algorithms for Perfect Copositive Matrices}
With a view towards the application of finding $\CP$-certificates, it makes sense to consider special algorithms for perfect copositive matrices.
We recall that we search systematically for perfect copositive matrices by traversing the neighborhood graph of the copositive Ryshkov polyhedron and that two contiguous perfect copositive matrices share a very large portion of their minimal vectors.
This very strong property has so far been unused and might offer significant speedups in the certificate algorithm of \autocite{CertificateAlgo}, if leveraged correctly.

Moreover, these speedups may also be beneficial for other applications.
One quite promising idea is that we can find polyhedral approximations for the complicated cone $\CP^n$ using perfect copositive matrices:
Taking a finite set $\mathcal{P}$ of perfect copositive matrices, we have an outer polyhedral approximation
\begin{equation*}
 \CP^n \subset \left\{ Q \in \NonN^n : \langle Q, P \rangle \geq 0 \text{ for all } P \in \mathcal{P} \right\}.
\end{equation*}
Using the Voronoi cone for each perfect copositive matrix of $\mathcal{P}$, we also have an inner polyhedral approximation
\begin{equation*}
 \cone \left\{ x x^T : x \in \MinCOP P \text{ for } P \in \mathcal{P} \right\} \subset \CP^n.
\end{equation*}
Both of these polyhedral approximations can be arbitrary close to $\CP^n$ (cf. \autocite{PerfectCopositive}).
Note that by duality we also get arbitrarily close linear approximations to $\COP^n$.
Since the certificate algorithm produces perfect copositive matrices as a byproduct, we may also use it to systematically generate such polyhedral approximations.
These approximations have not been investigated yet.
If they can be calculated sufficiently fast, they may prove very useful for solving or bounding copositive optimization problems.

\section*{Acknowledgement}
We like to thank Wessel van Woerden for helpful discussions on the complexity of the problem.
We also like to thank Thomas Kalinowski and Matthias Schymura for helpful comments on the manuscript.
Both authors gratefully acknowledge support by the German Research Foundation (DFG)
under grant SCHU 1503/8-1.

\printbibliography[keyword={ref}]
\printbibliography[keyword={software}, title={Software}]
\end{document}